\documentclass{amsart}
\usepackage{setspace}
\usepackage{a4}
\usepackage{amsthm,latexsym}
\usepackage{amsfonts}
\usepackage{graphicx}
\usepackage{textcomp}
\usepackage{cite}
\usepackage{enumerate}
\usepackage[mathscr]{euscript}
\usepackage{mathtools}
\newtheorem{theorem}{Theorem}[section]

\newtheorem{corollary}[theorem] {Corollary}
\newtheorem{definition}[theorem]{Definition}

\newtheorem{lemma} [theorem]{Lemma}

\newtheorem{problem}[theorem]{Problem}
\newtheorem{proposition}[theorem]{Proposition}
\newtheorem{remark}[theorem]{Remark}
\title{This is the title}
\usepackage{amssymb}
\usepackage{amsmath}
\usepackage{tikz}
\usepackage{hyperref}
\usepackage{mathtools}
\usetikzlibrary{cd}
\usepackage{fancyhdr}

\begin{document}
	\vspace{1cm}	
\begin{center}
		\vspace{0.3cm}	
	\hrule\hrule\hrule\hrule
		\vspace{0.3cm}	
{\bf{GROUP-FRAMES FOR BANACH SPACES}}\\
	\vspace{0.3cm}
\hrule\hrule\hrule\hrule
\vspace{0.3cm}
\textbf{K. MAHESH KRISHNA}\\
	School of Mathematics and Natural Sciences\\
	Chanakya University Global Campus \\
	NH-648, Haraluru Village\\
	Devanahalli Taluk, 	Bengaluru  North District\\
	Karnataka State, 562 110, India\\
	Email: kmaheshak@gmail.com\\
	Date: \today
\end{center}
\hrule
\vspace{0.5cm}
\textbf{Abstract}: In the literature, frames generated by unitary representations of groups (known as group-frames) have been studied only for Hilbert spaces. We conduct the first study of frames for Banach spaces generated by isometrically invertible representations of discrete groups on Banach spaces. These frames are characterized using left-regular and right-regular representations, Gram matrices, and group matrices on classical sequence spaces. A sufficiently large collection of functional–vector pairs using the double commutant of the representation is identified to generate group-frames for Banach spaces. Subsequently, we study Schauder frames generated by time–frequency shift operators on finite-dimensional Banach spaces. We derive the Moyal formula, the fundamental identity of Gabor analysis, the Wexler–Raz criterion, and the Ron–Shen duality in functional form.


\vspace{0.5cm}
\hrule
\tableofcontents
\hrule 
\section{Introduction}
 In their \textit{Memoirs} `\textbf{Frames, Bases and Group Representations}' \cite{HANMEMOIRS},  Han and Larson initiated the study of Parseval frames for Hilbert spaces generated by abstract countable groups, which  sheds light on the structure of unitary representations of a given group.  Recall that a countable group with a discrete topology is known as a discrete group. The definition of a Parseval frame  reads as follows.
\begin{definition}\cite{DUFFIN, HANMEMOIRS}\label{FRAMEDEF}
Let $G$ be a countable set. A sequence $\{\tau_g\}_{g\in G}$ in  a Hilbert space $\mathcal{H}$ is said to be a \textbf{Parseval frame} for $\mathcal{H}$ if 
\begin{align*}
\|h\|^2 = \sum_{g \in G} |\langle h, \tau_g\rangle|^2, \quad \forall h \in \mathcal{H}.
\end{align*}
\end{definition}
Definition \ref{FRAMEDEF}  is equivalent to the following equation: 

 \begin{align}\label{GENERALFOURIER}
 	h=\sum_{{g\in G}} \langle h, \tau_g\rangle \tau_g, \quad \forall h \in \mathcal{H}.
 \end{align}
One of the main objectives in frame theory is to generate Parseval frames from a given vector. In this regard, Han and Larson introduced the notion of frames generated by groups, defined as follows \cite{HANMEMOIRS}.
\begin{definition} \cite{HANMEMOIRS}\label{GROUPDEFHILBERT}
Let $G$ be a discrete group and $\{\tau_g\}_{g\in G}$ be a Parseval frame for a  Hilbert space $\mathcal{H}$. The frame $\{\tau_g\}_{g\in G}$ is said to be a \textbf{group-frame} if there exists a unitary representation $\pi$ of $G$ on $\mathcal{H}$ and a vector $\tau\in \mathcal{H}$   such that 
\begin{align*}
	\tau_g =\pi_g \tau, \quad \forall g \in G.
\end{align*}
In this case, the representation $\pi$  is called as \textbf{frame representation} and $\tau$ is called as a \textbf{frame vector}.
\end{definition}
After the introduction of group-frames, several important results appeared, which can be classified according to the types of groups as follows: 
\begin{enumerate}
	\item Finite groups \cite{HANCLASSIFICATION, VALEWALDRON2010, VALEWALDRON2016, WALDRONHAY, VALEWALDRON, ELDARBOLCSKEI, WALDRONBOOK, KORNELSON, CASAZZABOOK, LIHAN}.
	\item Discrete groups \cite{ROYSLAND, TANGWEBER, ALDROUBILARSONTANGWEBER, BARBIERIHERNANDEZPARCET}.
	\item locally compact groups \cite{BOWNIKIVERSON2021, CHRISTENSENGOH}.
	\item compact groups \cite{IVERSON}.
	\item Lie groups \cite{OUSSA, OUSSA2}.
	\item ICC groups \cite{DUTKAYHANPICIOROAGA}.
\end{enumerate}
Analogous to group-frames for Hilbert spaces, we may ask whether one can generate frames for Banach spaces using groups. In this paper, we are interested in frames for Banach spaces that are generated by discrete groups. Historically, it was Grochenig \cite{GROCHENIG} who introduced frames for Banach spaces known as Banach frames, which do not demand the reconstruction of an element using a series similar to the one  given in  Equation (\ref{GENERALFOURIER}). In 1999, Casazza, Han, and Larson \cite{CASAZZAHANLARSON}   introduced the notion of unconditional Schauder frames (also known as framings) focusing mainly on the expansion property given in Equation (\ref{GENERALFOURIER}).
 Let $\mathcal{X}$ be a separable Banach space and $\mathcal{X}^*$ be its dual. 
\begin{definition}\cite{CASAZZAHANLARSON}\label{FRAMING}
Let $G$ be a discrete group.	Let $\{\tau_g\}_{g\in G}$ be a sequence in  $\mathcal{X}$ and 	$\{f_g\}_{g\in G}$ be a sequence in  $\mathcal{X}^*.$ The pair $(\{f_g\}_{g\in G}, \{\tau_g\}_{g\in G})$ is said to be an \textbf{unconditional  Schauder frame} (we write USF) for $\mathcal{X}$ if 
	\begin{align}\label{SCHAUDEREQUATION}
		x=\sum_{g\in G}
		f_g(x)\tau_g, \quad \forall x\in \mathcal{X},
	\end{align}
where the series  in    (\ref{SCHAUDEREQUATION}) converges unconditionally.
\end{definition} 	
In this paper, we make the first attempt to understand USFs that are generated by  representations of discrete groups on Banach spaces. We organize the paper as follows. In Section \ref{2SECTION} we start by introducing the notion of group representations in Banach spaces. Then we introduce the notion of group-frames (Definition \ref{GROUPFRAMEDEF}) and show that this notion reduces to Definition \ref{GROUPDEFHILBERT} for Hilbert spaces with unitary representations. Subsequently, we introduce a  subclass of group-frames for Banach spaces called a group-p-unconditional Schauder frames (group-p-USFs), which factor through classical sequence spaces (Definition \ref{GROUPPUSF}). Theorem \ref{LEFTREGULAR} shows  that representations giving group-p-USFs  can be obtained by restricting the standard left-regular representation on sequence spaces. A connection between  group-p-USFs and group-matrices is derived in Theorem \ref{GMTHEOREM}. We show in Theorem \ref{IFFGROUP}  that  group-p-USFs can be characterized using algebraic equations.  Theorem \ref{GRAMIFFGROUP} gives a characterization of group-p-USFs using the standard left-regular representation on sequence spaces. Proposition \ref{SINGLECOM} and 
Theorem \ref{DOUBLECOM} show that once a pair of functional and a vector which generate  group-p-USF is non-empty, then the set of all such pairs is very large.\\
In Section \ref{SECTION3} we carry out finite dimensional Gabor analysis in functional form without using the inner product. Even though it is true that finite dimensional Banach spaces can be made a Hilbert space, rather than imposing an inner product, we believe that it is best to study by using functionals and vectors. The first important result of this section is the Banach space version of the Moyal formula, derived in Theorem \ref{MT}. Using this result, we obtain a large supply of Gabor-Schauder  frames for finite dimensional Banach spaces (Corollary \ref{GSCOR}). We then study Gabor-Schauder frames generated by subgroups of the group of all time-frequency shifts.
Theorem \ref{JANSSENTHEOREM} derives the fundamental identity of Gabor analysis for Banach spaces. The Wexler-Raz criterion for Gabor-Schauder frames  is derived in Theorem \ref{WEXLERRAZTHEOREM}. A partial Ron-Shen duality for Gabor-Schauder frames is derived in Theorem \ref{PARTIALRONSHEN}.
\begin{remark}
The paper has two parts:
\begin{enumerate}
	\item Part I: Abstract theory of group‑p‑USFs, including characterizations via regular representations, Gram matrices, and commutant methods.
	\item Part II: Finite Gabor‑Schauder frames on  for finite abelian groups, where we derive Banach‑space versions of classical Gabor identities.
\end{enumerate}
Although the two parts are conceptually distinct, both illustrate how group actions can generate structured Schauder frames in Banach spaces.	
\end{remark}

\section{Unconditional Schauder frames generated by groups}\label{2SECTION}
Throughout this paper, $G$ denotes a group (which need not be abelian). We denote the identity element of $G$ by $e$. Given a Banach space $\mathcal{X}$, let $\mathcal{II}(\mathcal{X})$ be the set of all invertible linear  isometries  on $\mathcal{X}$. The identity operator on $\mathcal{X}$ is denoted by $I_\mathcal{X}$.  We use the following definition of a representation in Banach spaces throughout.
\begin{definition}
	Let $\mathcal{X}$ be a Banach space and $G$ be a topological group. A mapping $\pi: G\to \mathcal{II}(\mathcal{X})$  is said to be an \textbf{invertible isometric  representation} of $G$ if the following conditions hold:
	\begin{enumerate}[\upshape(i)]
		\item $\pi: G\to \mathcal{II}(\mathcal{X})$ is a group homomorphism,  i.e., 
		\begin{align*}
			\pi_g\pi _h=\pi_{gh}, \quad \forall g, h \in G.
		\end{align*}
		\item $\pi$ is continuous in the following sense. For each fixed $x\in \mathcal{X}$, the map 
		\begin{align*}
			G\ni g \mapsto \pi_g x \in \mathcal{X}
		\end{align*}
		is continuous.
	\end{enumerate}
\end{definition}
In this section,  we consider only discrete  groups. Hence, condition (ii) is always satisfied. Given two invertible isometric representations $\pi: G\to \mathcal{II}(\mathcal{X})$ and $ \Delta:G \to \mathcal{II}(\mathcal{Y})$, we say that they are equivalent if there is an invertible operator $T:\mathcal{X} \to \mathcal{Y}$ which intertwines $\pi$ and $\Delta$. If the intertwining operator is an invertible isometry, then we say that the representations are invertibly isometric  equivalent.
For any discrete  group $G$ and $p\in [1,\infty)$, we  note that $G$ always admits two invertible isometric  representations on $\ell^p(G)$ defined as follows. Let $\{\delta_g\}_{g\in G}$ be the standard Schauder basis  for $\ell^p(G)$. We denote the coordinate functionals associated to $\{\delta_g\}_{g\in G}$ by $\{\zeta_g\}_{g\in G}$.
\begin{enumerate}[\upshape(i)]
	\item \textbf{$p$-left regular representation}    $\lambda: G \to \mathcal{II}(\ell^p(G))$ defined on $\{\delta_g\}_{g\in G}$ and extended linearly as 
	\begin{align*}
		\lambda_g \delta_h\coloneqq \delta_{gh}, \quad \forall g, h \in G.
	\end{align*}
	\item \textbf{$p$-right regular representation}    $\rho: G \to \mathcal{II}(\ell^p(G))$ defined on $\{\delta_g\}_{g\in G}$ and extended linearly as 
	\begin{align*}
		\rho_g \delta_h\coloneqq \delta_{hg^{-1}}, \quad \forall g, h \in G.
	\end{align*}
\end{enumerate}
With these preliminaries, we now set the following definition.
\begin{definition}\label{GROUPFRAMEDEF}
	Let $G$ be a discrete group and $(\{f_g\}_{g\in G}, \{\tau_g\}_{g\in G})$ be an USF  for a Banach space  $\mathcal{X}$. The USF $(\{f_g\}_{g\in G}, \{\tau_g\}_{g\in G})$ is said to be a \textbf{group-USF} if there exist an invertible isometric representation $\pi$ of $G$ on $\mathcal{X}$,  a vector $\tau\in \mathcal{X}$ and a functional  $f\in \mathcal{X}^*$ such that 
	\begin{align}\label{GFE}
	f_g=f\pi_{g^{-1}}, \quad 	\tau_g =\pi_g \tau, \quad \forall g \in G.
	\end{align}
\end{definition}
We first show that Definition  \ref{GROUPFRAMEDEF} truly generalizes Definition \ref{GROUPDEFHILBERT}. Let $\{\tau_g\}_{g\in G}=\{\pi_g\tau\}_{g\in G}$ be a group-frame for a Hilbert space $\mathcal{H}$. Let $f$ be the functional on $\mathcal{H}$ defined by $\tau$, i.e., $f(h)=\langle h, \tau \rangle $, for all $h\in \mathcal{H}$. Then 
\begin{align*}
	f_g(h)=	f(\pi_{g^{-1}}h)=\langle \pi_{g^{-1}}h, \tau\rangle =\langle h, \pi_g \tau\rangle =\langle h, \tau_g \rangle, \quad \forall h \in \mathcal{H}.
\end{align*}
Therefore $f_g$ is determined by $\tau_g$ for all $g\in G$.\\
   It seems that we can not construct a satisfactory theory for group-USF like that of group-frames for Hilbert spaces. Therefore we study a class of group-USF's defined as follows. Our definition is motivated by the notion of p-approximate Schauder frames introduced in \cite{MAHESHJOHNSON}.
\begin{definition}\label{GROUPPUSF}
Let $p\in [1, \infty)$. 	A group-USF $(\{f_g\}_{g\in G}, \{\tau_g\}_{g\in G})$ for a Banach space $\mathcal{X}$ is said to be a \textbf{group-p-USF} if the following conditions hold.
\begin{enumerate}[\upshape(i)]
	\item The map (analysis operator) 
		$
	\theta_f: \mathcal{X} \ni x \mapsto \{f_g(x) \}_{g\in G} \in \ell^p(G)
	$
	is a well-defined isometry. 
	\item The map (synthesis operator)	$
	\theta_\tau: \ell^p(G) \ni \{a_g \}_{g\in G} \mapsto \sum_{g\in G} a_g\tau_g \in \mathcal{X}
	$ 
	is a well-defined bounded linear operator.
\end{enumerate}
In this case, the representation $\pi$  is called as \textbf{p-USF representation} and the pair  $(f, \tau)$ is called as a \textbf{p-USF  functional-vector}.
\end{definition} 
A widely studied class of Banach spaces immediately following Hilbert spaces is that of the classical sequence spaces $\ell^p$ for $1\leq p <\infty$.. These spaces possess the best geometry next to Hilbert spaces, which is precisely why we consider this class.

We also need  following generalization of Definition \ref{GROUPPUSF}.
\begin{definition}
	Let $p\in [1, \infty)$ and $G$ be a discrete group. 	A USF $(\{f_g\}_{g\in G}, \{\tau_g\}_{g\in G})$ for a Banach space $\mathcal{X}$ is said to be a \textbf{p-USF} if the following conditions hold.
	\begin{enumerate}[\upshape(i)]
		\item The map (analysis operator) 
		$
		\theta_f: \mathcal{X} \ni x \mapsto \{f_g(x) \}_{g\in G} \in \ell^p(G)
		$
		is a well-defined isometry. 
		\item The map (synthesis operator)	$
		\theta_\tau: \ell^p(G) \ni \{a_g \}_{g\in G} \mapsto \sum_{g\in G} a_g\tau_g \in \mathcal{X}
		$ 
		is a well-defined bounded linear operator.
	\end{enumerate}
\end{definition} 
\begin{remark}
It is clear that a group-p-USF is a p-USF, and a group-USF is a USF. In this manuscript, the notion of a group-p-USF is our primary focus.
\end{remark}
We begin by presenting a characterization result for p-USFs, which is motivated by Holub's characterization of Hilbert space frames \cite{HOLUB}.
\begin{theorem}\label{THAFSCHAR}
	A pair  $(\{f_g\}_{g\in G}, \{\tau_g\}_{g\in G})$ is a p-USF   for 	$\mathcal{X}$, if and only if 
	\begin{align*}
		f_g=\zeta_g U, \quad \tau_g=V\delta_g, \quad \forall g \in G,
	\end{align*}  
	where $U:\mathcal{X} \rightarrow\ell^p(G)$, $ V: \ell^p(G)\to \mathcal{X}$ are bounded linear operators such that $VU=I_\mathcal{X}$ and $U$ is an isometry.
\end{theorem}
\begin{proof}
	$(\Leftarrow)$ Clearly $\theta_f$ and $\theta_\tau$ are bounded linear operators. Now let $x\in \mathcal{X}$. Then 
	\begin{align}\label{ORIGINALEQA}
		\sum_{{g\in G}}
		f_g(x)\tau_g=\sum_{{g\in G}} \zeta _g(Ux)V\delta _g=V\left(\sum_{{g\in G}} \zeta_g(Ux)\delta_g\right)=VUx=x.
	\end{align} 
	Note that, since $U$ is isometry, 
	\begin{align*}
		\|\theta_fx\|=\left\|\sum_{{g\in G}} f _g(x)\delta _g\right\|=\left\|\sum_{{g\in G}} \zeta _g(Ux)\delta _g\right\|=\|Ux\|=\|x\|, \quad \forall x \in \mathcal{X}.
	\end{align*}
	$(\Rightarrow)$ Define $U\coloneqq \theta_f$, $V\coloneqq \theta_\tau$. Then $\zeta_g(Ux)=\zeta_g(\theta_fx)=\zeta_g(\{f_k(x)\}_{k\in G})=f_g(x)$, $\forall x \in \mathcal{X}$, $V\delta_g=\theta_\tau \delta_g=\tau_g$, $\forall g \in G$ and $VU=\theta_\tau \theta_f=I_\mathcal{X}$. Since $\theta_f$ is an isometry, we also have $U$ is an isometry.
\end{proof}
We record the following important result from \cite{MAHESHJOHNSON}.
\begin{theorem}\cite{MAHESHJOHNSON}\label{OURS}
	Let  $(\{f_g\}_{g\in G}, \{\tau_g\}_{g\in G})$ be a p-USF for $\mathcal{X}$.  Then
	\begin{enumerate}[\upshape(i)]
		\item  $I_\mathcal{X}=\theta_\tau\theta_f.$
		\item  $G_{f, \tau}\coloneqq\theta_f\theta_\tau:\ell^p(G)\to \ell^p(G)$ is a projection onto   $\theta_f(\mathcal{X})$.
	\end{enumerate}
\end{theorem}
In \cite{HANMEMOIRS}, Han and Larson showed that, upto a unitary operator, a frame representation is a piece of the left regular representation. With the help of following lemma, we generalize their result to Banach spaces.
  \begin{lemma}\label{SUBLEMMA}
  	If $(\{f_g\}_{g\in G}, \{\tau_g\}_{g\in G})$ is a group-p-USF for $\mathcal{X}$, then range of its analysis operator is invariant under p-left-regular representation of $G$, i.e., $\lambda_g(\theta_f(\mathcal{X}))\subseteq \theta_f(\mathcal{X})$ for all $g\in G$.
  \end{lemma}
\begin{proof}
For any $g\in G$, 	
\begin{align*}
	\lambda_g \theta_f x&=\lambda_g\left(\sum_{h\in G}f_h(x)\delta_h\right)=\sum_{h\in G}f_h(x)\delta_{gh}=\sum_{u \in G}f_{g^{-1}u}(x)\delta_u\\
	&=\sum_{u \in G}f(\pi_{u^{-1}g}x)\delta_u
	=\sum_{u \in G}(f\pi_{u^{-1}})(\pi_gx)\delta_u=\theta_f(\pi_gx)\in \theta_f(\mathcal{X}), \quad \forall x \in \mathcal{X}.
\end{align*}
\end{proof}
  \begin{theorem}\label{LEFTREGULAR}
  	Every group-p-USF representation $\pi$ of $G$ is invertibly isometrically equivalent to a subrepresentation of  p-left-regular representation $\lambda$ of $G$.
  \end{theorem}
  \begin{proof}
 Let   $(\{f_g\}_{g\in G}, \{\tau_g\}_{g\in G})$ be  a group-p-USF for $\mathcal{X}$. Lemma \ref{SUBLEMMA} says that $\lambda_g(\theta_f(\mathcal{X}))\subseteq \theta_f(\mathcal{X})$ for all $g\in G$. Therefore the map 
 \begin{align*}
 	\Delta_g\coloneqq \lambda_g|_{\theta_f(\mathcal{X})}: \theta_f(\mathcal{X}) \to  \theta_f(\mathcal{X}), \quad \forall g \in G 
 \end{align*}
is a well-defined invertible isometric representation of $G$. We show that $ \Delta$ and $\pi$ are isometrically invertibly  equivalent. Note that $\theta_f: \mathcal{X} \to \theta_f(\mathcal{X})$ is an invertible isometry. We are done if we show that $\theta_f: \mathcal{X} \to \theta_f(\mathcal{X})$ intertwines $\Delta$ and $\pi$. This follows by doing a similar calculation as in the proof of  Lemma \ref{SUBLEMMA}.
  \end{proof}
\begin{remark}
	Theorem \ref{LEFTREGULAR} reduces to the Han-Larson result whenever $p=2$ and $\mathcal{X}$ is a Hilbert space.
\end{remark}
Vale and Waldron discovered that, for finite groups, groups-frames can be characterized using group-matrices \cite{VALEWALDRON}. We show that their result remains valid for Banach spaces. First, we recall the definition of group-matrix.
\begin{definition}\cite{JOHNSONBOOK}
	Let $G$ be a discrete group. A matrix $A\coloneqq [a_{g,h}]_{g,h \in G}$ over $\mathbb{C}$ is said to be a \textbf{group-matrix} if there exists a function $\nu: G \to  \mathbb{C}$ such that 
	\begin{align}\label{GROUPMATRIX}
		a_{g,h}=\nu(g^{-1}h), \quad \forall g,h \in G.
	\end{align}
\end{definition}
Let $(\{f_g\}_{g \in G}, \{\tau_g\}_{g \in G})$   be  a group-p-USF for $\mathcal{X}$. Then we note that 
\begin{align*}
f_g(\tau_h)	=(f\pi_{g^{-1}})(\pi_h\tau)=f(\pi_{g^{-1}h}\tau), \quad \forall g, h \in G. 
\end{align*}
Let $G_{f,\tau}\coloneqq\theta_f\theta_\tau$ be the Gramian of $(\{f_g\}_{g\in G}, \{\tau_g\}_{g\in G})$ whose matrix w.r.t.  the standard Schauder  basis $\{\delta_g\}_{g\in G}$ for $\ell^p(G)$ is given by 
\begin{align*}
	[ f_g(\tau_h) ]_{g,h \in G}.
\end{align*}
Now by defining $  \nu: G  \ni g \mapsto  \nu(g)\coloneqq f(\pi_g\tau) \in \mathbb{C}$,  we see from Equation (\ref{GROUPMATRIX}) that $G_{f, \tau}$ is a group-matrix. The next theorem shows that the converse of this also holds. 
  \begin{theorem}\label{GMTHEOREM}
   Let $G$ be a discrete group. Then a p-USF $(\{f_g\}_{g\in G}, \{\tau_g\}_{g\in G})$ for   $\mathcal{X}$	   is a group-p-USF  for $\mathcal{X}$ if and only if its Gramian $G_{f,\tau}$ is a group-matrix.	
  \end{theorem}
  \begin{proof}
  	As we already derived  only if part, we prove if part. Assume that the Gramian $G_{f, \tau}$ is a group-matrix. Then there exists a function $\nu: G \to  \mathbb{C}$ such that 
  	\begin{align*}
  		\nu (g^{-1}h)=f_g(\tau_h), \quad \forall g, h \in G.
  	\end{align*}  	
  	Given $g\in G$, define 
  	\begin{align*}
  		\pi_g: \mathcal{X}\ni x \mapsto \pi_g x \coloneqq \sum_{h\in G}f_{h}(x) \tau_{gh} \in \mathcal{X}
  	\end{align*}
  	and 
  	\begin{align*}
  		\Delta_g: \mathcal{X}\ni x \mapsto \Delta_g x \coloneqq \sum_{h\in G}f_{gh}(x) \tau_{h} \in \mathcal{X}.
  	\end{align*}
  	Note that 
  	\begin{align}\label{GROUPEQUATION}
  		f_{gu} (\tau_{gh})=\nu((gu)^{-1}gh)=\nu(u^{-1}h)= f_h(\tau_u) , \quad \forall u,g,h \in G.
  	\end{align}
  	Using Equation (\ref{GROUPEQUATION}), we get 
  	\begin{align*}
  	\Delta_g\pi_g x &=\Delta_g \left(\sum_{h\in G}f_h(x) \tau_{gh} \right)=\sum_{h\in G}f_h(x) \Delta_g\tau_{gh}=\sum_{h\in G}f_h(x)\sum_{u\in G} f_{gu}(\tau_{gh}) \tau_u\\
  	&=\sum_{h\in G}f_h(x)\sum_{u\in G} f_{u}(\tau_{h}) \tau_u=\sum_{h\in G}f_h(x)\tau_h=x,\quad \forall x \in \mathcal{X}
  	\end{align*}
  and 
  	\begin{align*}
  		\pi_g\Delta_gx &=	\pi_g\left( \sum_{h\in G}f_{gh}(x) \tau_{h}\right)=\sum_{h\in G}f_{gh}(x)\pi_g \tau_{h}=\sum_{h\in G}f_{gh}(x)\sum_{u\in G} f_u(\tau_h) \tau_{gu}\\
  		&=\sum_{h\in G}f_{gh}(x)\sum_{u\in G} f_{gu}(\tau_{gh}) \tau_{gu}=\sum_{h\in G}f_{gh}(x) \tau_{gh}=x,\quad \forall x \in \mathcal{X}.
  	\end{align*}
  Therefore $\Delta_g$ is the inverse of $\pi_g$. We next  show $\pi_g$ is isometry.
  	For $x \in \mathcal{X}$, using Equation (\ref{GROUPEQUATION}) and $\theta_f$ is an isometry, 
  	\begin{align*}
  		\|\pi_gx\|^p&=\|\theta_f(\pi_gx)\|^p=\sum_{h\in G}|f_h(\pi_gx)|^p=\sum_{h\in G}\left|f_h\left(\sum_{u \in G}f_u(x)\tau_{gu}\right)\right|^p\\
  			&=\sum_{h\in G}\left|\sum_{u \in G}f_u(x)f_h(\tau_{gu})\right|^p=\sum_{h\in G}\left|\sum_{u \in G}f_u(x)f_{gg^{-1}h}(\tau_{gu})\right|^p\\
  			&=\sum_{h\in G}\left|\sum_{u \in G}f_u(x)f_{g^{-1}h}(\tau_{u})\right|^p=\sum_{h\in G}\left|f_{g^{-1}h}\left(\sum_{u \in G}f_u(x)\tau_{u}\right)\right|^p\\
  			&=\sum_{h\in G}|f_{g^{-1}h}(x)|^p=\|x\|^p.
  	\end{align*}
   To show Equation (\ref{GFE})  again using Equation (\ref{GROUPEQUATION}), 
    \begin{align*}
   	\pi_g\tau_e&=\sum_{h\in G} f_h(\tau_e) \tau_{gh}=\sum_{u\in G} f_{g^{-1}u}(\tau_e) \tau_{u}=\sum_{u\in G} f_{g^{-1}u}(\tau_{g^{-1}g}) \tau_{u}\\
   	&=\sum_{u\in G} f_{u}(\tau_g) \tau_{u}=\tau_g, \quad \forall g \in G
   \end{align*}
and
\begin{align*}
	(f_e\pi_{g^{-1}})(x)&=f_e\left(\sum_{h\in G}f_h(x)\tau_{g^{-1}h}\right)=\sum_{h\in G}f_h(x)f_e(\tau_{g^{-1}h})=\sum_{h\in G}f_h(x)f_{g^{-1}g}(\tau_{g^{-1}h})\\
	&=\sum_{h\in G}f_h(x)f_g(\tau_{h})=f_g\left(\sum_{h\in G}f_h(x)\tau_{h}\right)
	=f_g(x),  \quad \forall g \in G, \forall x \in \mathcal{X}.
\end{align*}
We are left with showing that $\pi$ is a homomorphism. Let $g,h \in G$ and  $x \in \mathcal{X}$. Then using Equation (\ref{GROUPEQUATION}),
\begin{align*}
\pi_g\pi_hx&=\sum_{u\in G}f_u(\pi_hx)\tau_{gu}=\sum_{u\in G}f_u\left(\sum_{v\in G}f_v(x)\tau_{hv}\right)\tau_{gu}\\
&=\sum_{u\in G}\sum_{v\in G}f_v(x)f_u(\tau_{hv})\tau_{gu}=\sum_{u\in G}\sum_{v\in G}f_v(x)f_{hh^{-1}u}(\tau_{hv})\tau_{gu}\\
&=\sum_{u\in G}\sum_{v\in G}f_v(x)f_{h^{-1}u}(\tau_{v})\tau_{gu}=\sum_{u\in G}f_{h^{-1}u}\left(\sum_{v\in G}f_v(x)\tau_{v}\right)\tau_{gu}\\
&=\sum_{u\in G}f_{h^{-1}u}(x)\tau_{gu}=\sum_{v\in G}f_{v}(x)\tau_{(gh)v}=\pi_{gh}x.
\end{align*} 
  \end{proof}
\begin{remark}
	Theorem \ref{GMTHEOREM} reduces to the Vale-Waldron characterization result for group-frames  whenever $p=2$ and $\mathcal{X}$ is a Hilbert space.  
\end{remark}
Kaftal, Larson, and Zhang showed that group-frames can be characterized by an algebraic equation involving the inner product (actually, they proved this in the setting of operator-valued frames) \cite{KAFTALLARSONZHANG}. In the following result, we generalize the result of Kaftal, Larson, and Zhang to Banach spaces.
\begin{theorem}\label{IFFGROUP}
	Let $G$ be a discrete group and $(\{f_g\}_{g\in G}, \{\tau_g\}_{g\in G})$ be a p-USF for   $\mathcal{X}$.	 Then there is an invertible isometric representation  $\pi$ of $G$ on $\mathcal{X}$ for which 
	\begin{align}\label{UNITARYEQ}
		\tau_g=\pi_g \tau_e, \quad 
		f_g=f_e \pi_{g^{-1}}, \quad \forall g \in G
	\end{align}
	if and only if 
	\begin{align}\label{INNEREQUALS}
	f_{ug}(\tau_{uh})=f_g(\tau_h), \quad \forall u,g,h \in G.
	\end{align}
	Moreover, the  representation can be defined as 
	\begin{align}\label{REPDEFINITION}
		\pi_g\coloneqq\theta_\tau \lambda_g \theta_f, \quad \forall g \in G.
	\end{align}
\end{theorem}
\begin{proof}
	$(\Rightarrow)$ Let $ u,g,h \in G$. Then  we have 
	\begin{align*}
		f_{ug}(\tau_{uh})&=(f_e \pi_{(ug)^{-1}})(\pi_{uh} \tau_e)=(f_e \pi_{g^{-1}u^{-1}})(\pi_{uh} \tau_e)=(f_e\pi_{g^{-1}})(\pi_{h} \tau_e)=
f_g(\tau_h).
	\end{align*}
	$(\Leftarrow)$ Let $\pi_g$ be defined by Equation (\ref{REPDEFINITION}). We are required to show that   $\pi$ is an invertible isometric  representation and satisfies Equation (\ref{UNITARYEQ}). We first show that it satisfies Equation (\ref{UNITARYEQ}). To do so, we claim  the identity 
	\begin{align}\label{INBETWEEN}
		\lambda_g \theta_f \theta_\tau= \theta_f\theta_\tau\lambda_g,\quad \forall g \in G.
	\end{align}
Using Equation (\ref{INNEREQUALS})	we  verify Equation (\ref{INBETWEEN}). Note that it suffices to verify  Equation (\ref{INBETWEEN}) at the standard basis vectors $\delta_h$, $h\in G$.   Consider 
\begin{align*}
\theta_f\theta_\tau\lambda_g \delta_h&=\theta_f\theta_\tau \delta_{gh}=\theta_f\tau_{gh}=\sum_{u\in G} f_u(\tau_{gh})\delta_u=\sum_{u\in G} f_{gu}(\tau_{gh})\delta_{gu}=\sum_{u\in G} f_u(\tau_{h})\delta_{gu}\\
&=\sum_{u\in G} f_u(\tau_{gh})\lambda_g\delta_u=\lambda_g\left(\sum_{u\in G} f_u(\tau_{h})\delta_u\right)
=\lambda_g\theta_f\tau_h=\lambda_g \theta_f \theta_\tau\delta_h.
\end{align*}
Now using Equation (\ref{INBETWEEN})
	\begin{align*}
		\pi_g \tau_e=\theta_\tau \lambda_g \theta_f \tau_e =\theta_\tau \lambda_g \theta_f \theta_\tau \delta_e=\theta_\tau \theta_f \theta_\tau \lambda_g\delta_e=I_\mathcal{X}\theta_\tau \lambda_g\delta_e=\theta_\tau\delta_g=\tau_g, \quad \forall g \in G
			\end{align*}
		and 
		\begin{align*}
		f_e(\pi_{g^{-1}}x)&=f_e(\theta_\tau \lambda_{g^{-1}} \theta_fx)=f_e\left(\sum_{u\in G}f_u(x)\theta_\tau\lambda_{g^{-1}}\delta_u\right)=f_e\left(\sum_{u\in G}f_u(x)\tau_{g^{-1}u}\right)\\
		&=\sum_{u\in G}f_u(x)f_e(\tau_{g^{-1}u})
		=\sum_{u\in G}f_u(x)f_{g^{-1}g}(\tau_{g^{-1}u})=\sum_{u\in G}f_u(x)f_g(\tau_{u})\\
		&=f_g\left(\sum_{u\in G}f_u(x)\tau_u\right)=f_g(x), \quad \forall x \in \mathcal{X},  \forall g \in G.
	\end{align*}
	Now we show that $\pi$ is an invertible isometric representation. First we need to show that it is bijective. We note that, for $g\in G$, the operator $\Delta_g\coloneqq\theta_\tau \lambda_{g^{-1}} \theta_f$ is the inverse of $\pi_g$. In fact, using Equation (\ref{INBETWEEN}),
	\begin{align*}
		&\pi_g\Delta_g=\theta_\tau \lambda_{g} \theta_f\theta_\tau \lambda_{g^{-1}} \theta_f=\theta_\tau \theta_f\theta_\tau \lambda_g\lambda_{g^{-1}} \theta_f=I_\mathcal{X},\\
		&\Delta_g\pi_g=\theta_\tau \lambda_{g^{-1}} \theta_f\theta_\tau \lambda_{g} \theta_f=\theta_\tau \lambda_{g^{-1}}\lambda_g \theta_f\theta_\tau  \theta_f=I_\mathcal{X}.
	\end{align*}
	To show  $\pi_g$ is an isometry,  given  $x\in \mathcal{X}$,
		\begin{align*}
		\|\pi_gx\|^p&=\sum_{h\in G}|f_h(\pi_gx)|^p=\sum_{h\in G}|f_h(\theta_\tau\lambda_g\theta_fx)|^p=\sum_{h\in G}\left|(f_h\theta_\tau \lambda_g)\left(\sum_{u \in G}f_u(x)\delta_{u}\right)\right|^p\\
		&=\sum_{h\in G}\left|f_h\left(\sum_{u \in G}f_u(x)\theta_\tau \lambda_g \delta_u\right)\right|^p=\sum_{h\in G}\left|f_h\left(\sum_{u \in G}f_u(x)\tau_{gu}\right)\right|^p\\
		&=\sum_{h\in G}\left|\sum_{u \in G}f_u(x)f_h(\tau_{gu})\right|^p=\sum_{h\in G}\left|\sum_{u \in G}f_u(x)f_{gg^{-1}h}(\tau_{gu})\right|^p\\
		&=\sum_{h\in G}\left|\sum_{u \in G}f_u(x)f_{g^{-1}h}(\tau_{u})\right|^p=\sum_{h\in G}\left|f_{g^{-1}h}\left(\sum_{u \in G}f_u(x)\tau_{u}\right)\right|^p\\
		&=\sum_{h\in G}|f_{g^{-1}h}(x)|^p=\|x\|^p.
	\end{align*}
Finally, we  show that $\pi$ is representation: 
\begin{align*}
\pi_g\pi_h& =\theta_\tau \lambda_{g} \theta_f\theta_\tau \lambda_{h} \theta_f=\theta_\tau  \theta_f\theta_\tau \lambda_{g}\lambda_h \theta_f	
=I_\mathcal{X}\theta_\tau \lambda_{gh} \theta_f=\pi_{gh}, \quad \forall g, h \in G.
\end{align*}
\end{proof}
\begin{remark}
	Theorem \ref{IFFGROUP} reduces to the Kaftal-Larson-Zhang characterization for group-frames whenever $p=2$ and $\mathcal{X}$ is a Hilbert space.
\end{remark}
A careful observation on proof of the ``if" part of Theorem \ref{IFFGROUP} gives the following result. 
\begin{theorem}\label{LEFTREGULARCOMMUTES}
	Let $G$ be a discrete group and $(\{f_g\}_{g\in G}, \{\tau_g\}_{g\in G})$ be a group-p-USF for   $\mathcal{X}$. Then $	\lambda_g \theta_f \theta_\tau= \theta_f\theta_\tau \lambda_g,$ $ \forall g \in G.$
\end{theorem}
Next, we try to relate a frame representation to the p-right regular representation. This result was first derived by Mendez, Bodmann, Baker, Bullock and McLaney in the context of binary frame \cite{MENDEZBODMANNBAKERBULLOCKMCLANEY}.
\begin{theorem}\label{GRAMIFFGROUP}
	Let $G$ be a discrete group and $(\{f_g\}_{g\in G}, \{\tau_g\}_{g\in G})$ be a p-USF for   $\mathcal{X}$.  Then there is an invertible isometric representation  $\pi$ of $G$ on $\mathcal{X}$ for which 
	\begin{align*}
		\tau_g=\pi_g \tau_e,  \quad 
		f_g=f_e \pi_{g^{-1}}, \quad \forall g \in G
	\end{align*}
	if and only if 
	\begin{align*}
	\theta_f\theta_\tau	\{a_h\}_{h\in G}=\sum_{g\in G}\eta(g)\rho_g \{a_h\}_{h\in G}, \quad \forall \{a_h\}_{h\in G} \in \ell^p(G),
	\end{align*}
	where 
	\begin{align}\label{GRAMIANGROUP}
		\eta: G \ni g \mapsto \eta(g)\coloneqq f_e(\pi_{g}\tau_e) \in \mathbb{C}.
	\end{align}
	Moreover, the invertible isometric  representation can be defined as 
	\begin{align}\label{GRAMREP}
		\pi_g\coloneqq\theta_\tau \lambda_g \theta_f, \quad \forall g \in G.
	\end{align}
\end{theorem}
\begin{proof}
	$(\Rightarrow)$	Let $\eta $ be the function defined in Equation (\ref{GRAMIANGROUP}). Now for each $\delta_h$, 
	\begin{align*}
	\theta_f \theta_\tau \delta_h&=\theta_f\tau_h=\sum_{u\in G}f_u(\tau_h)\delta_u=\sum_{u\in G}(f\pi_{u^{-1}})(\pi_h\tau)\delta_u\\
	&=\sum_{u\in G}f(\pi_{u^{-1}h}\tau)\delta_u=\sum_{g\in G}f(\pi_{g}\tau)\delta_{hg^{-1}}=\sum_{g\in G}f(\pi_{g}\tau)\rho_g\delta_{h}\\
	&=\left(\sum_{g\in G}f(\pi_{g}\tau)\rho_g\right)\delta_{h}=\left(\sum_{g\in G}\eta(g)\rho_g\right)\delta_h.
	\end{align*}
		$(\Leftarrow)$  	 We   note that p-left and p-right regular representations commute. In fact, 	for any $g, h \in G$ and for each standard basis vector $\delta_u$, we have 
	\begin{align*}
		\lambda_g\rho_ h \delta_u=\lambda_g\delta_{uh^{-1}}=\delta_{guh^{-1}}=\rho_h\delta_{gu}=\rho_h\lambda_g\delta_{u}.
	\end{align*}
	We  then get 
	\begin{align*}
		\lambda_g \theta_f \theta_\tau \delta_h&=\lambda_g\left(\sum_{u \in G}\eta(u) \rho_u \delta_h\right)=\sum_{u \in G}\eta(u) \lambda_g\rho_u\delta_h	\\
		&=\left(\sum_{u \in G}\eta(u) \rho_u\right) \lambda_g\delta_h =\theta_f \theta_\tau\lambda_g \delta_h, \quad \forall g \in G.
	\end{align*}
Now we claim the following:  
\begin{align*}
	f_{ug}(\tau_{uh})=f_g(\tau_h), \quad \forall u,g,h \in G.
\end{align*}
Consider 
\begin{align*}
	f_{ug}(\tau_{uh})&=f_{ug}(\theta_\tau\delta_{uh})=\zeta_{ug}(\theta_f\theta_\tau\delta_{uh})=\zeta_{ug}\left(\sum_{v\in G}\eta(v)\rho_v \delta_{uh}\right)\\
	&=\sum_{v\in G}\eta(v)\zeta_{ug}\rho_v \delta_{uh}=\sum_{v\in G}\eta(v)\zeta_{ug} \delta_{uhv^{-1}}=\eta(g^{-1}h)=f_g(\tau_h), \quad \forall u,g,h \in G.
\end{align*}
Hence claim holds. 	For each $g\in G$, we now define $\pi_g$ as in Equation (\ref{GRAMREP}). Now by doing a similar calculation as in the converse part of proof of Theorem \ref{IFFGROUP} we get that $(\{f_g\}_{g\in G}, \{\tau_g\}_{g\in G})$ is a   group-p-USF for   $\mathcal{X}$. 
\end{proof}
\begin{remark}
Theorem \ref{GRAMIFFGROUP} is new even for Hilbert spaces, as the characterization of Mendez, Bodmann, Baker, Bullock, and McLaney is available only for binary frames.
\end{remark}
After giving several characterizations of frame representations, we next seek to determine the collection of functionals and vectors that generate group-frames. In the case of Hilbert spaces, Han and Larson completely characterized vectors which generate group-frames using the double commutant of the image of the representation \cite{HANMEMOIRS}. Even though we are unable to achieve this, we show that certain large sets generate group-frames for Banach spaces. We first set some notation. Let $G$ be a discrete group and let $\pi:G\to \mathcal{II}(\mathcal{X})$ be an invertible isometric representation. Assume that there is a vector $\tau\in \mathcal{X}$ and a functional $f \in \mathcal{X}^*$ such that $(\{f_g\}_{g \in G},\{\tau_g\}_{g \in G})$ is a group-p-USF for $\mathcal{X}$. Define the set of all \textbf{group-p-USF  vectors} as 
\begin{align*}
	\mathscr{F}_G(\pi, \mathcal{X})\coloneqq \{(f_1, \tau_1)  \in \mathcal{X}^*\times \mathcal{X}:  (\{{f_1}_g\}_{g \in G}, \{{\tau_1}_g\}_{g \in G}) \text{ is a group-p-USF  for } \mathcal{X}\}. 
\end{align*}
By assumption, $\mathscr{F}_G(\pi, \mathcal{X})\neq \emptyset$. We naturally ask what is the structure of $	\mathscr{F}_G(\pi, \mathcal{X})$? In the following proposition, we show that this set is quite large.  We use the following notation in the sequel. Given a subset $\mathcal{A}$ of linear operators on $\mathcal{X}$, by $\mathbb{I}^*(\mathcal{A})$ we mean the set of invertible isometric  operators $U:\mathcal{X}\to \mathcal{X}$ such that $U \in \mathcal{A}$. Given $f\in \mathcal{X}^*$ and $\tau \in \mathcal{X}$, we define 
\begin{align*}
&f[\mathbb{I}^*(\mathcal{A})^{-1},\mathbb{I}^*(\mathcal{A})]\tau \coloneqq \{(fU^{-1}, U\tau):U\in \mathbb{I}^*(\mathcal{A})\}.	
\end{align*}
\begin{proposition}\label{SINGLECOM}
If an invertible isometric  representation $\pi:G\to \mathcal{II}(\mathcal{X})$ admits a functional-frame vector $(f, \tau)$, then
	\begin{align}\label{COMMUATNTINEQUALITY}
f[\mathbb{I}^*(\pi(G)')^{-1},\mathbb{I}^*(\pi(G)')]\tau	\subseteq  \mathscr{F}_G(\pi, \mathcal{X}).
	\end{align}
\end{proposition}
\begin{proof}
	Let $(f_1, \tau_1)\in f[\mathbb{I}^*(\pi(G)')^{-1},\mathbb{I}^*(\pi(G)')]\tau$. Then $f_1=fU^{-1}$, $\tau_1=U\tau$ for some invertible isometry  $U:\mathcal{X}\to \mathcal{X}$ such that $U\in \mathbb{I}^*(\pi(G)')$. Define ${f_2}_g\coloneqq f_1 \pi_{g^{-1}}$ and  ${\tau_2}_g\coloneqq \pi_g\tau_1$, for all $g\in G$. Now we see that 
	\begin{align*}
	\sum_{g \in G}{f_2}_g(x){\tau_2}_g&=\sum_{g \in G}(f_1\pi_{g^{-1}})(x)\pi_g\tau_1=\sum_{g \in G}(fU^{-1})	(\pi_{g^{-1}}x)	\pi_g U\tau\\
		&=\sum_{g \in G}(f\pi_{g^{-1}})(U^{-1}x)	U\pi_g \tau=U\left(\sum_{g \in G}(f\pi_{g^{-1}})(U^{-1}x)\pi_g \tau\right)\\
		&=U\left(\sum_{g \in G}f_g(U^{-1}x)\tau_g\right)
		=UU^{-1}x=x, \quad \forall x\in \mathcal{X}.
	\end{align*}
	Therefore  $(\{{f_2}_g\}_{g \in G}, \{{\tau_2}_g\}_{g \in G})$  is a group-p-USF   for $\mathcal{X}$ and consequently $ (f_1, \tau_1)\in \mathscr{F}_G(\pi, \mathcal{X})$.
\end{proof}
We next try to show that $\mathscr{F}_G(\pi, \mathcal{X})$ contains another large set. For this, we need the following two results.  
\begin{theorem}\label{COMMUTATIONTHEOREM}
	For any group $G$, and any $1\leq  p<\infty$,
	\begin{align*}
		\lambda(G)'=\rho(G)''\quad 	\text{ and }\quad \rho(G)'=\lambda(G)''.
	\end{align*}	
\end{theorem}  
\begin{proof}
	We already know  from the proof of Theorem \ref{GRAMIFFGROUP} that p-left and p-right regular representations commute, i.e., 
	\begin{align*}
		\lambda_g\rho_ h =\rho_h\lambda_g, \quad \forall g, h \in G.
	\end{align*}
	Hence  for every $g\in G$, $\lambda_g\in  \rho(G)'$. By varying $g$, we get $\lambda(G)\subseteq \rho(G)'$. Taking commutant yield $\lambda(G)'\supseteq \rho(G)''$. Now we prove the reverse inclusion. Let $ T\in \lambda(G)'$. To show $T\in  \rho(G)''$ we need to show that $TS=ST$ for all $S\in  \rho(G)'$. So let $S\in  \rho(G)'$. Note that to verify $TS=ST$, it suffices to verify $TS\delta_h=ST\delta_h$ for all $h \in G$. Let $h \in G$. Then 
		\begin{align*}
		TS\delta_h&=TS \rho_{h^{-1}}\delta_e=T \rho_{h^{-1}}S\delta_e=T \rho_{h^{-1}}\left(\sum_{g \in G}\zeta_g(S\delta_e) \delta_g\right)\\
		&=\sum_{g \in G}\zeta_g(S\delta_e) T \rho_{h^{-1}}\delta_g=\sum_{g \in G}\zeta_g(S\delta_e) T \delta_{gh}=\sum_{g \in G}\zeta_g(S\delta_e) T \lambda_{gh}\delta_{e}\\
		&=\sum_{g \in G}\zeta_g(S\delta_e) \lambda_{gh}T\delta_{e}=\sum_{g \in G}\zeta_g(S\delta_e) \lambda_{gh}\left(\sum_{u \in G} \zeta_u(T\delta_e) \delta_u\right)\\
		&=\sum_{g \in G}\zeta_g(S\delta_e) \sum_{u \in G}\zeta_u(T\delta_e) \lambda_{gh}\delta_u=\sum_{g \in G}\sum_{u \in G}\zeta_g(S\delta_e)\zeta_u(T\delta_e)\delta_{ghu}
	\end{align*}
	and 
	\begin{align*}
		ST\delta_h&=ST\lambda_h\delta_e	=S\lambda_hT\delta_e=S\lambda_h\left(\sum_{u \in G}\zeta_u(T\delta_e) \delta _u\right)	\\
		&=\sum_{u \in G}\zeta_u(T\delta_e) S\lambda_h\delta _u=\sum_{u \in G}\zeta_u(T\delta_e) S\delta _{hu}=\sum_{u \in G}\zeta_u(T\delta_e) S\rho _{(hu)^{-1}}\delta_e\\
		&=\sum_{u \in G} \zeta_u(T\delta_e) \rho _{(hu)^{-1}}S\delta_e=\sum_{u \in G}\zeta_u(T\delta_e) \rho _{(hu)^{-1}}\left(\sum_{g\in G}\zeta_g(S\delta_e) \delta_g\right)\\
		&=\sum_{u \in G}\zeta_u(T\delta_e)\sum_{g\in G}\zeta_g(S\delta_e)  \rho _{(hu)^{-1}}\delta_g=\sum_{u \in G}\sum_{g \in G}\zeta_u(T\delta_e)\zeta_g(S\delta_e) \delta_{ghu}.
	\end{align*}
	Therefore $\lambda(G)'=\rho(G)''$. Finally $\lambda(G)''=\rho(G)'''=\rho(G)'$.
\end{proof}  
\begin{theorem}\label{JAJTHEOREM}
	For any discrete group $G$ and for every $1\leq p<\infty$, the map 
	\begin{align*}
		\Phi: \rho(G)'' \ni A \mapsto \Phi(A)\coloneqq JAJ\in \lambda(G)''
	\end{align*}
	is an algebra isomorphism, where 
	\begin{align*}
		J: \ell^p(G) \ni \{a_g\}_{g\in G}  \mapsto J \{a_g\}_{g\in G} \coloneqq \sum_{g \in G}a_g\delta_{g^{-1}}\in \ell^p(G).
	\end{align*}
	Moreover, we have the following. 
	\begin{enumerate}[\upshape(i)]
		\item If $U\in \rho(G)''$ is invertible (resp. isometry), then $\Phi(U)$ is invertible (resp. isometry).
	\item If $V\in \lambda(G)''$ is invertible (resp. isometry), then  $\Phi^{-1}(V)$ is invertible (resp. isometry).
	\end{enumerate}
\end{theorem}
\begin{proof}
	We first note that $J$ is  an isomorphism and $J^2=I_{\ell^p(G)}$. Before showing $\Phi$ is an isomorphism, we need to show that it is well-defined. Let $A\in  \rho(G)''$. We try to show that $JAJ\in \lambda(G)''$ which says $\Phi$ is well-defined. By Theorem  \ref{COMMUTATIONTHEOREM}, showing $JAJ\in \lambda(G)''$ is same as showing $JAJ\in \rho(G)'$.  Let $g\in G$ be arbitrary. We claim that $JAJ\rho_g=\rho_gJAJ$. Since $\{\delta_h\}_{h\in G}$ is a basis for $\ell^p(G)$ it suffices to show $JAJ\rho_g\delta_h=\rho_gJAJ\delta_h$, for all $h \in G$. Now noting $A \in \lambda(G)'$, we get 
		\begin{align*}
		JAJ\rho_g\delta_h&=JAJ\delta_{hg^{-1}}=JA\delta_{gh^{-1}}=JA\lambda_g \delta_{h^{-1}}=J\lambda_gA \delta_{h^{-1}}\\
		&=J\lambda_g\left(\sum_{u\in G}\zeta_u (A\delta_{h^{-1}}) \delta_u\right)=\sum_{u\in G}\zeta_u (A\delta_{h^{-1}}) J\lambda_g\delta_u\\
		&=\sum_{u\in G}\zeta_u (A\delta_{h^{-1}}) J\delta_{gu}=\sum_{u\in G}\zeta_u (A\delta_{h^{-1}}) \delta_{u^{-1}g^{-1}}
	\end{align*}
	and 
		\begin{align*}
		\rho_gJAJ\delta_h&=	\rho_gJA\delta_{h^{-1}}=\rho_gJ\left(\sum_{u\in G}\zeta_u (A\delta_{h^{-1}}) \delta_u\right)=\sum_{u\in G}\zeta_u (A\delta_{h^{-1}}) \rho_gJ\delta_u
		\\
		&=\sum_{u\in G}\zeta_u (A\delta_{h^{-1}}) \rho_g\delta_{u^{-1}}=\sum_{u\in G}\zeta_u (A\delta_{h^{-1}}) \delta_{u^{-1}g^{-1}}.
	\end{align*}
	Clearly $\Phi$ is linear. Since $J^2=I_{\ell^p(G)}$, $\Phi$ is multiplicative. Through a direct calculation we see that inverse of $\Phi$ is the map $\lambda(G)'' \ni B \mapsto  JBJ\in \rho(G)''$.
\end{proof}
The following is the most important result on generators of group-frames for Banach spaces.
\begin{theorem}\label{DOUBLECOM}
	Let $\pi:G\to \mathcal{II}(\mathcal{X})$ be a invertible isometric  representation of a group $G$ 	which admits a functional-vector $(f,\tau)$. Then 
	\begin{align}\label{DCINEQUALITY}
f[\mathbb{I}^*(\pi(G)'')^{-1},\mathbb{I}^*(\pi(G)'')]\tau	\subseteq  \mathscr{F}_G(\pi, \mathcal{X}).
	\end{align}
\end{theorem}
\begin{proof}
	Let $g\in G$ and define  $\Delta_g \coloneqq \lambda_g|_{{\theta_f(\mathcal{X})}}$. Theorem \ref{LEFTREGULAR} says that $\pi$ is invertible isometrically  equivalent to the representation $\Delta$ with functional-vector $(\zeta_e(\theta_f\theta_\tau), \theta_f\theta_\tau\delta_e)$.  Therefore, without loss of generality we may assume that 
	\begin{align*}
		\mathcal{X}= \theta_f(\mathcal{X}), \quad \pi=\Delta, \quad 	G=\{\lambda _g|{_{\theta_f(\mathcal{X})}}: g \in G\}, \quad \tau=\theta_f \theta_\tau \delta_e, \quad f=\zeta_e(\theta_f\theta_\tau).
	\end{align*}
	Let $U:\theta_f(\mathcal{X})\to \theta_f(\mathcal{X})$ be an invertible isometry  such that $U\in 	\mathbb{I}^*(\pi(G)'')$. We need to show that $(fU^{-1}, U\tau)\in\mathscr{F}_G(\pi, \mathcal{X})$. Define $\tilde{\tau}\coloneqq U\tau$,   $\tilde{\tau_g}\coloneqq \pi_g \tilde{\tau}$ and $\tilde{f}\coloneqq fU^{-1}$, $\tilde{f_g}\coloneqq \tilde{f}\pi_{g^{-1}}$ for all $g \in G$.   To prove the theorem, now it suffices to show that $(\{\tilde{f_g}\}_{g \in G}, \{\tilde{\tau_g}\}_{g \in G})$   is an  group-p-USF  for $\mathcal{X}$. Let $x=\theta_fy\in \theta_f(\mathcal{X})$. Using  Theorem \ref{COMMUTATIONTHEOREM} and   Theorem \ref{JAJTHEOREM}, we may assume that $U\in \{\lambda _g|{_{\theta_f(\mathcal{X})}}: g \in G\}'$. Since $\theta_f\theta_\tau z=z$ for all $z \in  \theta_f(\mathcal{X})$ and $U:\theta_f(\mathcal{X})\to \theta_f(\mathcal{X})$, we then have  $\theta_f \theta_\tau  U= U$. 
 Therefore
	\begin{align*}
		\sum_{{g\in G}}\tilde{f_g}(x)\tilde{\tau_g}&=\sum_{{g\in G}}(\tilde{f}\pi_{g^{-1}})(x)\pi_g\tilde{\tau}=\sum_{{g\in G}}((fU^{-1})\pi_{g^{-1}})(x)\pi_gU\tau\\
		&=\sum_{{g\in G}}f(U^{-1}\pi_{g^{-1}}x)\pi_gU\tau=\sum_{{g\in G}}(\zeta_e\theta_f\theta_\tau)(U^{-1}\pi_{g^{-1}}x)\pi_gU\theta_f \theta_\tau \delta_e\\
		&=\sum_{{g\in G}}\zeta_e(\theta_f\theta_\tau U^{-1}\pi_{g^{-1}}x)\pi_gU\theta_f \theta_\tau \delta_e=\sum_{{g\in G}}\zeta_e(\theta_f\theta_\tau U^{-1}\lambda_{g^{-1}}x)\lambda_gU\theta_f \theta_\tau \delta_e\\
		&=\sum_{{g\in G}}\zeta_e(\theta_f\theta_\tau \lambda_{g^{-1}}U^{-1}x)U\lambda_g\theta_f \theta_\tau \delta_e=\sum_{{g\in G}}\zeta_e(\lambda_{g^{-1}}\theta_f\theta_\tau U^{-1}x)U\theta_f \theta_\tau \lambda_g\delta_e\\
		&=\sum_{{g\in G}}\zeta_g(\theta_f\theta_\tau U^{-1}x)U\theta_f \theta_\tau \delta_g=U\theta_f \theta_\tau\theta_f\theta_\tau U^{-1}x\\
		&=U\theta_f \theta_\tau U^{-1}x=UU^{-1}x=x.
	\end{align*}
	\end{proof}
\begin{remark}
	Theorem \ref{DOUBLECOM} reduces to the Han-Larson result for group-frames whenever $p=2$ and $\mathcal{X}$ is a Hilbert space.
\end{remark}
Han and Larson characterized frame vectors for Hilbert spaces by means of an equality (see Theorem 6.17 in \cite{HANMEMOIRS}). Later, Kaftal, Larson and Zhang noticed that the set of all frame vectors is path-connected in the norm topology (see Theorem 8.1 in \cite{KAFTALLARSONZHANG}). Based on these results, we raise the following questions.
\begin{problem}
	Given an invertible isometric  representation $\pi:G\to \mathbb{I}^*(\mathcal{X})$ which  admits a functional-vector $(f,\tau)$, 	characterize $\mathscr{F}_G(\pi, \mathcal{X})$. In particular, classify Banach  spaces, groups and representations  such that 
	\begin{enumerate}[\upshape(i)]
		\item  	$f[\mathbb{I}^*(\pi(G)')^{-1},\mathbb{I}^*(\pi(G)')]\tau=  \mathscr{F}_G(\pi, \mathcal{X})$.
		\item $	f[\mathbb{I}^*(\pi(G)'')^{-1},\mathbb{I}^*(\pi(G)'')]\tau=  \mathscr{F}_G(\pi, \mathcal{X})$.
	\end{enumerate}
\end{problem}  
\begin{problem}
Is $\mathscr{F}_G(\pi, \mathcal{X}) \subseteq \mathcal{X}^*\times \mathcal{X}$ path connected in the  norm-topology?
\end{problem}

\section{Finite Gabor-Schauder frames}\label{SECTION3}
In this section, we study Schauder frames generated by time-frequency shifts on finite abelian groups for the finite dimensional Banach space $\mathbb{C}^{o(G)}$. Our main motivation comes from the Gabor analysis on finite abelian groups presented by Pfander in \cite{PFANDER} and by Feichtinger, Kozek and Luef in  \cite{FEICHTINGERKOZEKLUEF}. Let $G$ be a finite abelian group with identity $e$, $o(G)$ be the order of $G$ and $\widehat{G}$ be the set of all characters of $G$. We denote the circle group by $\mathbb{T}$. Then $\widehat{G}$ becomes a group with respect to pointwise multiplication of characters. The character which sends every element of $G$ to 1 is called the identity character and is denoted by $1_G$. Let $\{\delta_g\}_{g\in G}$ be the standard basis for $\mathbb{C}^{o(G)}$ and $\{\zeta_g\}_{g\in G}$ be the coordinate functionals associated with $\{\delta_g\}_{g\in G}$. For $x=(x_g)_{g\in G}, y=(y_g)_{g\in G}\in  \mathbb{C}^{o(G)}$, we set 
\begin{align*}
	x \cdot y \coloneqq (x_gy_g)_{g\in G}
\end{align*}
and 
\begin{align*}
	x^*\coloneqq (\overline{x_g})_{g\in G}.
\end{align*}
From classical Fourier analysis on finite abelian groups, we note the following properties
 (see \cite{NATHANSON}). 
\begin{enumerate}
	\item $o(G)=o(\widehat{G})$.
	\item If $\xi, \chi \in \widehat{G}$, then $\langle \xi, \chi \rangle =\delta_{ \xi, \chi }$, where
	\begin{align}\label{CHARACTERSORTHOGONAL}
		\langle \xi, \chi \rangle \coloneqq \frac{1}{o(G)}\sum_{g\in G}\xi(g)\overline{\chi(g)}.
	\end{align}
	\item If $g,h \in G$, then $\langle g, h \rangle =\delta_{g,h}$, where 
	\begin{align}\label{VECTORSORTHOGONAL}
		\langle g, h \rangle \coloneqq \frac{1}{o(\widehat{G})}\sum_{\xi\in \widehat{G}}\xi(g)\overline{\xi(h)}.
	\end{align}
\end{enumerate}

Given $\xi \in \widehat{G}$,  the  \textbf{modulation operator} $M_\xi: \mathbb{C}^{o(G)}\to \mathbb{C}^{o(G)}$ is defined by 
\begin{align*}
	M_\xi(x_g)_{g\in G}\coloneqq (\xi(g)x_g)_{g\in G}, \quad \forall (x_g)_{g\in G}\in \mathbb{C}^{o(G)}.
\end{align*}
Given $k\in G$,  the  \textbf{translation operator} $T_k: \mathbb{C}^{o(G)}\to \mathbb{C}^{o(G)}$ by 
\begin{align*}
	T_k(x_g)_{g\in G}\coloneqq (x_{g-k})_{g\in G}, \quad \forall (x_g)_{g\in G}\in \mathbb{C}^{o(G)}.
\end{align*}
Let $\lambda\coloneqq (k, \xi) \in G\times \widehat{G}$.  By composing  modulation operator and  translation operator we get the  \textbf{time-frequency shift operator} $\pi(k, \xi) : \mathbb{C}^{o(G)}\to \mathbb{C}^{o(G)}$ defined as 

\begin{align*}
	\pi(\lambda)=	\pi(k, \xi)\coloneqq 	M_\xi T_k.
\end{align*}
Thus the action of $  \pi(\lambda)$ is given by 
\begin{align*}
	\pi(\lambda)(x_g)_{g\in G}=(\xi(g)x_{g-k})_{g\in G}, \quad \forall (x_g)_{g\in G}\in \mathbb{C}^{o(G)}.
\end{align*}
Following properties of time-frequency shift operators will be used and are well-known.
\begin{theorem}\cite{FEICHTINGERKOZEKLUEF}   \label{COMMTIMEFREQUENCY}
	Let 	$\lambda = (k, \xi), \mu=(l, \chi) \in G\times \widehat{G}$. Then 
	\begin{enumerate}[\upshape(i)]
		\item $\pi(\lambda+\mu)=\chi(k)\pi(\lambda)\pi(\mu)$.
		\item $\pi(\lambda)\pi(\mu)=\overline{\chi(k)}\xi(l)\pi(\mu)\pi(\lambda)$.
		\item $\pi(\lambda)^{-1}=\overline{\xi(k)}\pi(-\lambda)$.
	\end{enumerate}
\end{theorem}
Motivated by discrete Gabor analysis over finite abelian groups (see \cite{PFANDER, FEICHTINGERKOZEKLUEF}), we set the following notion. We emphasize here that, even though $\mathbb{C}^{o(G)}$ is a Hilbert space, once a Banach space structure is given on it, thinking of $\mathbb{C}^{o(G)}$ as a Hilbert space in frame theory will not work. A recent influential instance is the definition of the notion of `\textbf{Frame Potential}' for Banach spaces, where the usual direct generalization of the Hilbert space frame potential failed (see Proposition 2.5 in \cite{CHAVEZDOMINGUEZFREEMANKORNELSON}).
\begin{definition}
	Given a subgroup  $\Lambda \subseteq  G\times \widehat{G}$, a nonzero $f\in (\mathbb{C}^{o(G)})^*$ and a nonzero $\tau \in \mathbb{C}^{o(G)}$, the pair $(\{f(\pi(\lambda)^{-1})\}_{\lambda \in \Lambda}, \{\pi(\lambda)\tau\}_{\lambda \in \Lambda})$ is called as a \textbf{Gabor-Schauder system}. If the operator 
	\begin{align*}
S_{f,\tau, \Lambda}:	\mathbb{C}^{o(G)}\ni x \mapsto \sum_{\lambda\in \Lambda}f(\pi(\lambda)^{-1}x)\pi(\lambda)\tau \in 	\mathbb{C}^{o(G)}
	\end{align*}
 is invertible, then 	$(\{f(\pi(\lambda)^{-1})\}_{\lambda \in \Lambda}, \{\pi(\lambda)\tau\}_{\lambda \in \Lambda})$ is called as a \textbf{Gabor-Schauder frame}. 
\end{definition}
We are interested in subgroups $\Lambda$ of $G\times \widehat{G}$, which will give Gabor-Schauder frames. First we show that the full lattice $ G\times \widehat{G}$ will give a Gabor-Schauder frame.  Given a nonzero $f\in (\mathbb{C}^{o(G)})^*$, we define analysis operator 
\begin{align*}
	W_f:\mathbb{C}^{o(G)}\ni x \mapsto W_f x \coloneqq (f(\pi(\lambda)^{-1}x)_{\lambda \in G \times \widehat{G}}\in \mathbb{C}^{o(G\times \widehat{G})}
\end{align*}
and given a nonzero $\tau \in \mathbb{C}^{o(G)}$, we define synthesis operator 
\begin{align*}
	V_\tau: \mathbb{C}^{o(G\times \widehat{G})} \ni (a_\lambda)_{\lambda \in G \times \widehat{G}}\mapsto \sum_{\lambda\in G \times \widehat{G}} a_\lambda \pi(\lambda)\tau \in \mathbb{C}^{o(G)}.
\end{align*}
Our first result is that composition of previous two operators give scalar times identity.
 \begin{theorem}\label{MT}
 For  $f\in (\mathbb{C}^{o(G)})^*$ and 	$\tau \in \mathbb{C}^{o(G)}$, 
 \begin{align}\label{MOYAL}
 	V_\tau W_f=o(G)f(\tau)I_{\mathbb{C}^{o(G)}}.
 \end{align}
 \end{theorem}
\begin{proof}
Let $x=(x_g)_{g\in G} \in \mathbb{C}^{o(G)}$ and $\tau=(\tau_g)_{g\in G}$. Then 

\begin{align*}
	&V_\tau W_fx=V_\tau (f[\pi(\lambda)^{-1}x])_{\lambda \in G \times \widehat{G}}=\sum_{\lambda\in G \times \widehat{G}}f[\pi(\lambda)^{-1}x]\pi(\lambda)\tau\\
	&=	\sum_{(k, \xi)\in G \times \widehat{G}}f[\overline{\xi(k)}\pi(-k, -\xi)(x_g)_{g\in G}]\pi(k, \xi)(\tau_g)_{g\in G}\\
	&=\sum_{(k, \xi)\in G \times \widehat{G}}f[\overline{\xi(k)}((-\xi)(h)x_{h+k})_{h\in G})] (\xi(g)\tau_{g-k})_{g\in G}=\sum_{(k, \xi)\in G \times \widehat{G}}f[(\overline{\xi(k+h)}x_{h+k})_{h\in G}](\xi(g)\tau_{g-k})_{g\in G}\\
	&=\left(\sum_{(k, \xi)\in G \times \widehat{G}}f[(\overline{\xi(k+h)}x_{h+k})_{h\in G}]\xi(g)\tau_{g-k}\right)_{g\in G}=\left(\sum_{(k, \xi)\in G \times \widehat{G}}f[(\overline{\xi(k+h)}\xi(g)x_{h+k})_{h\in G}]\tau_{g-k}\right)_{g\in G}\\
		&=\left(\sum_{k\in G}\sum_{\xi\in\widehat{G}}f[(\overline{\xi(k+h)}\xi(g)x_{h+k})_{h\in G}]\tau_{g-k}\right)_{g\in G}
		=\left(\sum_{k\in G}f\left[\sum_{\xi\in\widehat{G}}(\overline{\xi(k+h)}\xi(g)x_{h+k})_{h\in G}\right]\tau_{g-k}\right)_{g\in G}\\
		&=\left(\sum_{k\in G}f\left[\sum_{\xi\in\widehat{G}}\overline{\xi(k+h)}\xi(g)x_{h+k}\right]_{h\in G}\tau_{g-k}\right)_{g\in G}
		=o(G)\left(\sum_{k\in G}f\left[\delta_{k+h,g}x_{h+k}\right]_{h\in G}\tau_{g-k}\right)_{g\in G}\\
		&=o(G)\sum_{g\in G}\sum_{k\in G}f\left[\delta_{k+h,g}x_{h+k}\right]_{h\in G}\tau_{g-k}\delta_g
		=o(G)\sum_{g\in G}\sum_{k\in G}f\left[\sum_{h\in G}\delta_{k+h,g}x_{h+k}\delta_h\right]\tau_{g-k}\delta_g\\
		&=o(G)\sum_{g\in G}\sum_{k\in G}\sum_{h\in G}\delta_{k+h,g}x_{h+k}f[\delta_h]\tau_{g-k}\delta_g=o(G)\sum_{g\in G}\sum_{k\in G}x_gf[\delta_{g-k}]\tau_{g-k} \delta_{g}\\
		&=o(G)\sum_{g\in G}x_gf\left[\sum_{k\in G}\tau_{g-k}\delta_{g-k}\right]\delta_g=o(G)\sum_{g\in G}x_gf[(\tau_k)_{k\in G}]\delta_g=o(G)f(\tau)x.
\end{align*}	
\end{proof}
 We call Equation (\ref{MOYAL}) the \textbf{Schauder-Moyal formula} for Banach spaces. It is easy to see that for Hilbert spaces, whenever $f$ is determined by $\tau$, the Schauder-Moyal formula reduces to the familiar Moyal formula. The Schauder-Moyal formula immediately gives the following corollary.
\begin{corollary}\label{GSCOR}
 If $f\in (\mathbb{C}^{o(G)})^*$ and 	$\tau \in \mathbb{C}^{o(G)}$ are such that $f(\tau)\neq 0$, then 	
\begin{align}\label{INVERSIONFORMULA}
	x=\frac{1}{o(G)f(\tau)}\sum_{\lambda\in G \times \widehat{G}}f(\pi(\lambda)^{-1}x)\pi(\lambda)\tau, \quad \forall x \in \mathbb{C}^{o(G)}.
\end{align}	
In other words,  
\begin{align*}
	(\{f(\pi(\lambda)^{-1})\}_{\lambda \in G\times \widehat{G}}, \{\pi(\lambda)\tau\}_{\lambda \in G\times \widehat{G}})
\end{align*}
is a Gabor-Schauder frame for 	$\mathbb{C}^{o(G)}$.
\end{corollary}
We call Equation (\ref{INVERSIONFORMULA}) the \textbf{inversion formula for Banach spaces}. It reduces to the inversion formula for the short-time Fourier transform for Hilbert spaces whenever $f$ is determined by $\tau$.\\
Recall that given linear operators $T, S:\mathbb{C}^{o(G)} \to \mathbb{C}^{o(G)}$, if we define 
\begin{align*}
	\langle T, S \rangle_{\text{HS}}\coloneqq \sum_{g \in G} \langle T\delta_g, S\delta_g \rangle, 
\end{align*}
then  the space $\mathcal{L}(\mathbb{C}^{o(G)})$  of all linear operators from $\mathbb{C}^{o(G)}$ to itself    is a Hilbert space w.r.t. inner product $	\langle T, S \rangle_{\text{HS}}$. We denote this Hilbert space  by $\text{HS}(\mathbb{C}^{o(G)})$. We need the  following result in continuation. 
\begin{theorem}\cite{FEICHTINGERKOZEKLUEF}\label{ONBHS}
	The family 
	\begin{align}\label{HSBASIS}
		\left\{\frac{1}{\sqrt{o(G)}}\pi(\lambda) \right\} _{\lambda\in G\times \widehat{G}}	
	\end{align}
	is an  orthonormal basis  for $\text{HS}(\mathbb{C}^{o(G)})$.
\end{theorem}
Recall that given a subgroup (also known as lattice) $\Lambda \subseteq  G\times \widehat{G}$, we define the  adjoint subgroup  of $\Lambda $, denoted by $\Lambda ^0$ as 
\begin{align*}
	\Lambda ^0	\coloneqq \{\mu \in G\times \widehat{G}: \pi(\lambda)\pi(\mu)= \pi(\mu)\pi(\lambda), \forall \lambda \in \Lambda\}.
\end{align*}
Theorem  \ref{COMMTIMEFREQUENCY}  says that  $\Lambda ^0$ is a subgroup of $G\times \widehat{G}$.  Now given a nonzero vector $\tau \in \mathbb{C}^{o(G)}$ and a nonzero functional $f\in  (\mathbb{C}^{o(G)})^*$, we define the  \textbf{Gabor-Schauder  frame operator} $S_{f, \tau, \Lambda }:\mathbb{C}^{o(G)} \to \mathbb{C}^{o(G)}$ as 
\begin{align*}
	S_{f,\tau, \Lambda}x\coloneqq \sum_{\lambda\in \Lambda}f(\pi(\lambda)^{-1}x)\pi(\lambda)\tau, \quad \forall x \in \mathbb{C}^{o(G)}.
\end{align*}
Following key result will be used repeatedly. It mainly uses group properties of $\Lambda$.
\begin{theorem}\label{GABORCOMMUTES}
	Let  	$\Lambda $ be a subgroup of $G\times \widehat{G}$,   $\tau    \in \mathbb{C}^{o(G)}$ and $f\in  (\mathbb{C}^{o(G)})^*$. Then for each $\mu \in  \Lambda$, the  time-frequency shifts $\pi(\mu)$ commute with the Gabor-Schauder frame operator   $S_{f, \tau, \Lambda}$. 
\end{theorem}
\begin{proof}
	Since $\pi(\mu)$ is invertible, to show $\pi(\mu)S_{f, \tau,  \Lambda}=S_{f, \tau,  \Lambda}\pi(\mu)$ it suffices to show that $\pi(\mu)^{-1} S_{f,\tau, \Lambda}\pi(\mu)=S_{f, \tau, \Lambda}$. Let $x \in \mathbb{C}^{o(G)}$ and $\mu=(l,\chi) $. Now using  the fact that $G$ is a group and using Theorem  \ref{COMMTIMEFREQUENCY},
	\begin{align*}
	\pi(l,\chi) ^{-1} S_{f, \tau,  \Lambda}\pi(l,\chi) x&=\pi(l,\chi) ^{-1} \left(\sum_{(k, \xi)\in \Lambda }f[\pi(k, \xi)^{-1}\pi(l,\chi) x]   \pi (k,\xi)\tau\right)\\
	&=\sum_{(k, \xi)\in \Lambda }f[\pi(k, \xi)^{-1}\pi(l,\chi) x]   \pi(l,\chi) ^{-1}\pi (k,\xi)\tau\\
	&=\sum_{(k, \xi)\in \Lambda }f[(\pi(l,\chi)^{-1}\pi(k, \xi))^{-1} x]   \pi(l,\chi) ^{-1}\pi (k,\xi)\tau\\
	&=\sum_{(k, \xi)\in \Lambda }f[(\overline{\chi(l)}\pi(-l,-\chi)\pi(k, \xi))^{-1} x]   \overline{\chi(l)}\pi(-l,-\chi) \pi (k,\xi)\tau\\
	&=\sum_{(k, \xi)\in \Lambda }f[(\pi(-l,-\chi)\pi(k, \xi))^{-1} x]   \pi(-l,-\chi) \pi (k,\xi)\tau\\
	&=\sum_{(k, \xi)\in \Lambda }f[(\overline{\xi(-l)}\pi(-l+k,-\chi+\xi))^{-1} x]   \overline{\xi(-l)}\pi(-l+k,-\chi+\xi)\tau\\
	&=\sum_{(k, \xi)\in \Lambda }f[\pi(-l+k,-\chi+\xi)^{-1} x]   \pi(-l+k,-\chi+\xi)\tau=S_{f, \tau,  \Lambda}x.
	\end{align*}
\end{proof}
Recall that (see\cite{FREEMANODELLSCHLUMPRECHT}) a pair  $(\{f_j\}_{j=1}^n, \{\tau_j\}_{j=1}^n)$ is said to be an approximate Schauder frame (written as ASF) for $\mathbb{C}^{o(G)}$   if the operator 
\begin{align*}
	S_{f, \tau}: \mathbb{C}^{o(G)} \ni x \mapsto S_{f, \tau}x\coloneqq \sum_{j=1}^{n}f_j(x)\tau_j \in  \mathbb{C}^{o(G)}
\end{align*}
is invertible. Also we recall that an ASF $(\{g_j\}_{j=1}^n, \{\omega_j\}_{j=1}^n)$  for $\mathbb{C}^{o(G)}$ is said to be a dual for $(\{f_j\}_{j=1}^n, \{\tau_j\}_{j=1}^n)$  if 

\begin{align*}
	x=\sum_{j=1}^{n}f_j(x)\omega_j=\sum_{j=1}^{n}g_j(x)\tau_j, \quad \forall x \in \mathbb{C}^{o(G)}.
\end{align*}
It is an easy observation that the ASF $(\{f_jS_{f, \tau}^{-1}\}_{j=1}^n, \{S_{f, \tau}^{-1}\tau_j\}_{j=1}^n)$ is always a dual for $(\{f_j\}_{j=1}^n, \{\tau_j\}_{j=1}^n)$. This is called as the canonical dual of $(\{f_j\}_{j=1}^n, \{\tau_j\}_{j=1}^n)$. Following corollary says that the canonical dual of Gabor-Schauder frame is again a Gabor-Schauder frame.
\begin{corollary}\label{CANONICALDUALGABOR}
	Let  	$\Lambda $ be a subgroup of $G\times \widehat{G}$,   $\tau    \in \mathbb{C}^{o(G)}$ and $f\in  (\mathbb{C}^{o(G)})^*$. If $(\{f(\pi(\lambda)^{-1})\}_{\lambda \in \Lambda}, \{\pi(\lambda)\tau\}_{\lambda \in \Lambda})$ is a   Gabor-Schauder   frame for $ \mathbb{C}^{o(G)}$, then its canonical dual can be written as $(\{\phi(\pi(\lambda)^{-1})\}_{\lambda \in \Lambda}, \{\pi(\lambda)\omega\}_{\lambda \in \Lambda})$ for some $\omega  \in \mathbb{C}^{o(G)}$ and $\phi\in  (\mathbb{C}^{o(G)})^*$. In other words,     there exist  $\omega  \in \mathbb{C}^{o(G)}$ and $\phi \in  (\mathbb{C}^{o(G)})^*$ such that $(\{\phi(\pi(\lambda)^{-1})\}_{\lambda \in \Lambda}, \{\pi(\lambda)\omega\}_{\lambda \in \Lambda})$ is a  Gabor-Schauder   frame for $ \mathbb{C}^{o(G)}$ and 
	\begin{align*}
		x=	\sum_{\lambda\in \Lambda}\phi(\pi(\lambda)^{-1}x) \pi(\lambda)\tau=	\sum_{\lambda\in \Lambda}f(\pi(\lambda)^{-1}x)  \pi(\lambda)\omega, \quad \forall x \in \mathbb{C}^{o(G)}.
	\end{align*}
\end{corollary}
\begin{proof}
	Canonical dual for $(\{f(\pi(\lambda)^{-1})\}_{\lambda \in \Lambda}, \{\pi(\lambda)\tau\}_{\lambda \in \Lambda})$ is given by $(\{f(\pi(\lambda)^{-1}S_{f,\tau, \Lambda}^{-1})\}_{\lambda \in \Lambda}, \{S_{f,\tau, \Lambda}^{-1}\pi(\lambda)\tau\}_{\lambda \in \Lambda})$.  Define $\omega\coloneqq S_{f, \tau, \Lambda}^{-1}\tau$ and $\phi\coloneqq fS_{f, \tau, \Lambda}^{-1}$. Then using Theorem \ref{GABORCOMMUTES}, 
	
	\begin{align*}
		&\{ S_{f, \tau, \Lambda}^{-1}\pi (\lambda)\tau\}_{\lambda \in \Lambda}=\{ \pi(\lambda)S_{f, \tau, \Lambda}^{-1}\tau\}_{\lambda \in \Lambda}=\{ \pi (\lambda)\omega\}_{\lambda \in \Lambda}, \\
		&\{f(\pi(\lambda)^{-1}S_{f, \tau, \Lambda}^{-1})\}_{\lambda \in \Lambda}=\{f(S_{f, \tau, \Lambda}^{-1}\pi(\lambda)^{-1})\}_{\lambda \in \Lambda}=\{\phi(\pi(\lambda)^{-1})\}_{\lambda \in \Lambda}.
	\end{align*}
\end{proof}
Next result is the key result of this section.
\begin{theorem}\label{JANSSENTHEOREM}
	Let  	$\Lambda $ be  a subgroup of $G\times \widehat{G}$ and  $\tau  \in \mathbb{C}^{o(G)}$, $f\in  (\mathbb{C}^{o(G)})^*$. Then 
	\begin{align}\label{BEFOREJANSSEN}
		S_{f, \tau,  \Lambda}x=	\sum_{\lambda\in \Lambda}f(\pi(\lambda)^{-1}x) \pi(\lambda)\tau    =\frac{o(\Lambda)}{o(G)}\sum_{\mu\in \Lambda^0}f(\pi(\mu)^{-1}\tau) \pi(\mu)x=\frac{o(\Lambda)}{o(G)}S_{f,  x, \Lambda^0}\tau, \quad \forall x \in \mathbb{C}^{o(G)}.
	\end{align}
\end{theorem}
\begin{proof}
	Theorem \ref{ONBHS} says that 	
	\begin{align*}
		S_{f, \tau, \Lambda}&=\sum_{\mu \in G\times \widehat{G}}\left\langle 	S_{f, \tau,  \Lambda}, \frac{1}{\sqrt{o(G)}}\pi(\mu)\right \rangle _\text{HS}\frac{1}{\sqrt{o(G)}}\pi(\mu)\\
		&=\sum_{\mu \in G\times \widehat{G}}c_\mu\pi(\mu), \quad \text{ where } c_\mu\coloneqq \frac{1}{o(G)}\left\langle 	S_{f, \tau,  \Lambda}, \pi(\mu)\right \rangle _\text{HS}, \forall \mu \in G\times \widehat{G}.
	\end{align*}
	Let $\mu \in G\times \widehat{G}$ and $\lambda \in \Lambda$. Theorem  \ref{COMMTIMEFREQUENCY} then gives a $d_{\mu, \lambda}\in \mathbb{T}$ such that $\pi(\lambda)^{-1}\pi(\mu)\pi(\lambda)=d_{\mu, \lambda}\pi(\mu)$.  Using Theorem \ref{GABORCOMMUTES}, 
	\begin{align*}
		\sum_{\mu \in G\times \widehat{G}}c_\mu\pi(\mu)&=\pi(\lambda)^{-1} S_{f, \tau,  \Lambda}\pi (\lambda)=\pi(\lambda)^{-1}\left(\sum_{\mu \in G\times \widehat{G}}c_\mu\pi(\mu)\right)\pi (\lambda)\\
		&=\sum_{\mu \in G\times \widehat{G}}c_\mu\pi(\lambda)^{-1}\pi(\mu)\pi(\lambda)=\sum_{\mu \in G\times \widehat{G}}c_\mu d_{\mu, \lambda}\pi(\mu).
	\end{align*}
	Therefore
	\begin{align*}
		\sum_{\mu \in G\times \widehat{G}}c_\mu(1-d_{\mu, \lambda})\pi(\mu)=0, \quad \forall \lambda \in \Lambda.
	\end{align*}
	Since $\{\pi(\lambda)\}_{\lambda \in G\times \widehat{G}}$ is linearly independent (Theorem \ref{ONBHS}), we then have $c_\mu(1-d_{\mu, \lambda})=0$, for all $\lambda \in \Lambda, $ for all  $\mu \in G\times \widehat{G}$. Let $\mu\notin \Lambda^0$. We claim that $c_\mu=0$. If this is not true, then $1-d_{\mu, \lambda}=0$ for all $\lambda \in \Lambda$. But then we have $\pi(\mu)\pi(\lambda)=\pi(\lambda)\pi(\mu)$ for all $\lambda \in \Lambda$ which says $\mu\in \Lambda^0$ which is impossible. So claim holds. Hence the formula for the Gabor-Schauder  frame operator reduces to 
	\begin{align}\label{JANSSENPROOF}
		S_{f, \tau,  \Lambda}=	\sum_{\mu \in \Lambda^0}c_\mu\pi(\mu).
	\end{align}
	Now let $\mu \in \Lambda^0$. Then using Equation (\ref{HSBASIS}), 
	
	\begin{align*}
		c_\mu&=	\frac{1}{o(G)}\left\langle 	S_{f, \tau,  \Lambda}, \pi(\mu)\right \rangle _\text{HS}=	\frac{1}{o(G)}\sum_{g \in G}\langle S_{f, \tau,  \Lambda}\delta_g, \pi(\mu) \delta_g\rangle\\
		&=	\frac{1}{o(G)}\sum_{g \in G}\left\langle  \sum_{\lambda\in \Lambda}f(\pi(\lambda)^{-1}\delta_g)  \pi(\lambda)\tau , \pi(\mu) \delta_g\right \rangle 
		=	\frac{1}{o(G)} \sum_{\lambda\in \Lambda}\sum_{g \in G}f(\pi(\lambda)^{-1}\delta_g) \langle\pi(\lambda)\tau , \pi(\mu) \delta_g \rangle\\
		&=\frac{1}{o(G)} \sum_{\lambda\in \Lambda}\sum_{g \in G}f(\pi(\lambda)^{-1}\delta_g) \langle\pi(\mu)^{-1}\pi(\lambda)\tau ,  \delta_g \rangle
		=\frac{1}{o(G)} \sum_{\lambda\in \Lambda}\sum_{g \in G}f\left[\langle\pi(\mu)^{-1}\pi(\lambda)\tau ,  \delta_g \rangle\pi(\lambda)^{-1}\delta_g\right]\\
			&=\frac{1}{o(G)} \sum_{\lambda\in \Lambda} f\left[\sum_{g \in G}\langle\pi(\mu)^{-1}\pi(\lambda)\tau ,  \delta_g \rangle\pi(\lambda)^{-1}\delta_g\right]=\frac{1}{o(G)} \sum_{\lambda\in \Lambda} f\left[\pi(\lambda)^{-1}\sum_{g \in G}\langle\pi(\mu)^{-1}\pi(\lambda)\tau ,  \delta_g \rangle\delta_g\right]\\
			&=\frac{1}{o(G)} \sum_{\lambda\in \Lambda} f(\pi(\lambda)^{-1}\pi(\mu)^{-1}\pi(\lambda)\tau)=\frac{1}{o(G)} \sum_{\lambda\in \Lambda} f(\pi(\mu)^{-1}\pi(\lambda)^{-1}\pi(\lambda)\tau)\\
			&=\frac{1}{o(G)} \sum_{\lambda\in \Lambda} f(\pi(\mu)^{-1}\tau)=\frac{o(\Lambda)}{o(G)}f(\pi(\mu)^{-1}\tau).
	\end{align*}
	By substituting the value of $c_\mu$ in Equation (\ref{JANSSENPROOF}) finally gives 
	\begin{align*}
		S_{f, \tau,  \Lambda}=\frac{o(\Lambda)}{o(G)}	\sum_{\mu \in \Lambda^0}f(\pi(\mu)^{-1}\tau) \pi(\mu).	
	\end{align*}
\end{proof}
We call Equation (\ref{BEFOREJANSSEN}) the \textbf{Fundamental Identity of Gabor-Schauder Frames (FIGSF)} or the \textbf{Schauder-Wexler-Raz identity} or the \textbf{Schauder-Janssen representation} of the Gabor-Schauder frame operator $S_{f, \tau, \Lambda}$. It represents the frame operator corresponding to $\Lambda$ in terms of the frame operator corresponding to $\Lambda^0$.  For Hilbert spaces, it reduces to the usual Fundamental Identity of Gabor Analysis (famously known as FIGA) derived in \cite{PFANDER} and \cite{FEICHTINGERKOZEKLUEF}. The first major consequence of FIGSF is the following criterion.
\begin{theorem}\label{WEXLERRAZTHEOREM}
	Let  	$\Lambda $ be  a subgroup of $G\times \widehat{G}$ and  $\tau  \in \mathbb{C}^{o(G)}$,  $f\in  (\mathbb{C}^{o(G)})^*$. Then 
	\begin{align}\label{ALGEBRAICWEXLERRAZ}
		x=\sum_{\lambda\in \Lambda}f(\pi(\lambda)^{-1}x) \pi(\lambda)\tau  , \quad \forall x \in \mathbb{C}^{o(G)}	\iff 
		f(\pi(\mu)^{-1}\tau) =\frac{o(G)}{o(\Lambda)}\delta_{\mu,0}, \quad \forall \mu\in \Lambda^0.
	\end{align}	
\end{theorem} 
\begin{proof}
	$(\Rightarrow)$ We have that the Gabor-Schauder frame operator $S_{f, \tau, \Lambda}$ is identity operator. Theorem \ref{JANSSENTHEOREM} then says that 
	\begin{align*}
		I_{\mathbb{C}^{o(G)}}=\frac{o(\Lambda)}{o(G)}\sum_{\mu\in \Lambda^0}f(\pi(\mu)^{-1}\tau) \pi(\mu)
		=\frac{o(\Lambda)}{o(G)}f(\tau)	I_{\mathbb{C}^{o(G)}}+\frac{o(\Lambda)}{o(G)}\sum_{\mu\in \Lambda^0\setminus \{(0,0)\}}f(\pi(\mu)^{-1}\tau) \pi(\mu).
	\end{align*}
	By Theorem \ref{ONBHS}, the collection $	\{\pi(\lambda) \} _{\lambda\in G\times \widehat{G}}$	   is linearly independent. Thus the validity of previous equation gives 
	
	\begin{align*}
		f(\tau) =\frac{o(G)}{o(\Lambda)}, \quad f(\pi(\mu)^{-1}\tau) =0, \forall \mu\in \Lambda^0\setminus \{(0,0)\}.
	\end{align*}
	$(\Leftarrow)$ Using Equation (\ref{BEFOREJANSSEN}), 
	\begin{align*}
		S_{f, \tau,  \Lambda}x=\frac{o(\Lambda)}{o(G)}\sum_{\mu\in \Lambda^0}f(\pi(\mu)^{-1}\tau) \pi(\mu)x
		=\frac{o(\Lambda)}{o(G)}f(\tau)x+\frac{o(\Lambda)}{o(G)}\sum_{\mu\in \Lambda^0\setminus \{(0,0)\}}f(\pi(\mu)^{-1}\tau) \pi(\mu)x=x, \quad \forall x \in \mathbb{C}^{o(G)}.
	\end{align*}	
\end{proof}
We call Equation (\ref{ALGEBRAICWEXLERRAZ}) the \textbf{Schauder-Wexler-Raz criterion}. This is a generalization of the Wexler-Raz criterion derived in \cite{PFANDER} and \cite{FEICHTINGERKOZEKLUEF}. The criterion says when we can say that certain pairs give Gabor-Schauder frames by checking an algebraic equation on the adjoint lattice. We conclude by deriving the following result, which we call the \textbf{partial Schauder Ron-Shen duality}. Note that for Hilbert spaces, the theorem is an ''if and only if' statement (see \cite{PFANDER} and \cite{FEICHTINGERKOZEKLUEF}).
\begin{theorem}\label{PARTIALRONSHEN}
	Let  	$\Lambda $ be  a subgroup of $G\times \widehat{G}$ and  $\tau \in \mathbb{C}^{o(G)}$, $f\in  (\mathbb{C}^{o(G)})^*$. If $(\{f(\pi(\lambda)^{-1})\}_{\lambda \in \Lambda}, \{\pi(\lambda)\tau\}_{\lambda \in \Lambda})$ is a Gabor-Schauder frame for $ \mathbb{C}^{o(G)}$, then both  sets $\{ \pi (\mu)^{-1}\tau\}_{\mu \in \Lambda^0}$ and  $\{f(\pi(\mu)^{-1})\}_{\mu \in \Lambda^0}$ are  linearly independent.
\end{theorem} 
\begin{proof}
Without loss of generality we may assume that $S_{f, \tau, \Lambda}=I_{\mathbb{C}^{o(G)}}.$ Now Theorem  \ref{WEXLERRAZTHEOREM}   says that 
	\begin{align*}
		f(\pi(\mu)^{-1}\tau) =\frac{o(G)}{o(\Lambda)}\delta_{\mu,0}, \quad \forall \mu\in \Lambda^0.
	\end{align*}
	Let $\lambda, \mu \in \Lambda^0$ be such that $\lambda\neq  \mu$. Then, since $ \Lambda^0$ is a group,  we get that $	f(\pi(\lambda)^{-1}\pi(\mu)^{-1}\tau) =0$. We now suppose that $\sum_{\mu \in \Lambda^0}c_\mu \pi(\mu)^{-1}\tau=0$ for some $c_\mu \in \mathbb{C}$. Then for each $\lambda \in \Lambda^0$, we have 
	\begin{align*}
		0=f\left(\pi(\lambda)^{-1}\left(\sum_{\mu \in \Lambda^0}c_\mu \pi(\mu)^{-1}\tau\right)\right)=\sum_{\mu \in \Lambda^0}c_\mu f(\pi(\lambda)^{-1}\pi(\mu)^{-1}\tau)=c_\lambda \frac{o(G)}{o(\Lambda)}.
	\end{align*}
	Therefore  $\{ \pi (\lambda)\tau\}_{\lambda \in \Lambda^0}$ is linearly independent. On the other hand, let $\sum_{\mu \in \Lambda^0}d_\mu f(\pi(\mu)^{-1})=0$ for some $d_\mu \in \mathbb{C}$. Then for each $\lambda \in \Lambda^0$, we have 
	\begin{align*}
		0=\sum_{\mu \in \Lambda^0}d_\mu f(\pi(\mu)^{-1}\pi(\lambda)^{-1} \tau)=d_\lambda \frac{o(G)}{o(\Lambda)}.
	\end{align*}
	Therefore   $\{f(\pi(\mu)^{-1})\}_{\mu \in \Lambda^0}$ is linearly independent.
\end{proof}
\begin{remark}
It is true that every norm is equivalent to every other norm in a finite-dimensional Banach space, and hence equivalent to the Hilbert space norm. However, the main issue is that the constants establishing the equivalence of these norms are usually difficult to determine or unknown. This is the reason why we conducted our study without assuming any Hilbert space structure.
\end{remark}

  \section{Acknowledgments}
  Author thanks Prof. B.V. Rajarama Bhat from Indian Statistical Institute, Bangalore Centre for several useful discussions and suggestions which improved the manuscript. The author received financial support from Indian Statistical Institute, Bangalore Centre through the J. C. Bose Fellowship of Prof. B. V. Rajarama Bhat. 
  The author thanks the anonymous reviewers for their careful review of the manuscript and valuable suggestions, which greatly improved the article.

 \bibliographystyle{plain}
 \bibliography{reference.bib}

\end{document}